\RequirePackage{fix-cm}
\documentclass[11pt, a4paper, oneside, DIV=11, pagebackref]{amsart}
\usepackage[utf8]{inputenc}
\usepackage[T1]{fontenc}
\usepackage{amssymb}
\usepackage{mathtools}
\usepackage[headinclude=true]{typearea}
\usepackage[hidelinks, linktoc=all]{hyperref}
\usepackage[type={CC}, modifier={by}, version={4.0}]{doclicense}
\usepackage[stretch=10]{microtype}

\usepackage{tikz}
\usetikzlibrary{decorations.pathreplacing,arrows.meta}

\providecommand{\href}[2]{#2}

\providecommand*{\backref}{}
\providecommand*{\backrefalt}{}
\renewcommand*{\backref}[1]{}
\renewcommand*{\backrefalt}[4]{%
	\ifcase #1 %
	\or
	  Cited page~#2.
	\else
	  Cited pages~#2.
	\fi
}

\newcommand\MTkillspecial[1]{
  \bgroup
  \catcode`\&=9
  \let\\\relax%
  \scantokens{#1}%
  \egroup
}
\newcommand\DeclarePairedDelimiterMultiline[3]{
  \DeclarePairedDelimiter{#1}{#2}{#3}
  \reDeclarePairedDelimiterInnerWrapper{#1}{star}{
    \mathopen{##1\vphantom{\MTkillspecial{##2}}\kern-\nulldelimiterspace\right.}
    ##2
    \mathclose{\left.\kern-\nulldelimiterspace\vphantom{\MTkillspecial{##2}}##3}}
}

\newcommand{\Z}{\mathbb{Z}}
\newcommand{\N}{\mathbb{N}}
\newcommand{\R}{\mathbb{R}}

\newcommand{\dd}{\mathop{}\!\mathrm{d}}
\DeclarePairedDelimiterMultiline{\abs}{\lvert}{\rvert}
\DeclarePairedDelimiterMultiline{\norm}{\lVert}{\rVert}
\DeclarePairedDelimiterMultiline{\pare}{(}{)}
\newcommand{\st}{\::\:}

\DeclareMathOperator{\supp}{supp}

\DeclareMathOperator{\id}{id}
\DeclareMathOperator{\Leb}{Leb}

\newcommand{\coloneqq}{\mathrel{\mathop:}=}

\renewcommand{\epsilon}{\varepsilon}

\renewcommand{\phi}{\varphi}
\renewcommand{\leq}{\leqslant}
\renewcommand{\geq}{\geqslant}

\newtheorem{thm}{Theorem}[section]
\newtheorem{prop}[thm]{Proposition}
\newtheorem{definition}[thm]{Definition}
\newtheorem{lem}[thm]{Lemma}

\newtheorem*{prop*}{Proposition}

\theoremstyle{definition}

\newtheorem{rmk}[thm]{Remark}

\numberwithin{equation}{section}

\title[a.e.~pair of iets does not generate a free group]{almost every pair of interval exchange transformations does not
generate a free group}
\author{Sébastien Gouëzel}

\address{IRMAR, CNRS UMR 6625,
Université de Rennes, 35042 Rennes, France}
\email{sebastien.gouezel@univ-rennes.fr}

\date{\today}

\begin{document}

\begin{abstract}
We show that the Lebesgue measure of the set of pairs of interval exchange
transformations that generate a free group of rank $2$ is zero.
\end{abstract}

\maketitle

\section{Introduction}

While individual interval exchange transformations have been studied at depth
and are now pretty well understood~\cite{veech_metric_1, viana_iet,
yoccoz_iet}, the global structure of the group of interval exchange
transformations is much more mysterious. Some authors wonder if it could be
amenable (see~\cite{juschenko_mattebon_monod_delasalle_iet, garban_iet}), but
it is not even known whether it contains a free group. This is a famous
question of Katok. Several partial results are known showing that there is no
free group in several
situations.~\cite{juschenko_mattebon_monod_delasalle_iet} shows that there is
no free group when the translation vectors in $\R/\Z$ generate a group of
rank at most $2$, and~\cite{dahmani_fujiwara_guirardel_iet_free, jay_iet}
show that topologically generic pairs of elements do not generate a free
group. See also~\cite{dahmani_fujiwara_guirardel_solvable_iet,
guelman_liousse}. The difficulty is that the orbit of a single point in the
interval can grow at most polynomially as translations commute, but many
orbits can exhibit different behaviors giving rise to faster global growth.
For instance, there is a free semigroup of interval exchange transformations,
by~\cite[Theorem~8.1]{dahmani_fujiwara_guirardel_iet_free}

The open set of pairs of interval exchange transformations produced
by~\cite{dahmani_fujiwara_guirardel_iet_free, jay_iet} in which there is a
relation is obtained by very small perturbations around rational ones. In
particular, one expects them to have very small Lebesgue measure. Therefore,
in view of the dichotomy between topological genericity and measure
genericity that appears throughout dynamical systems, a natural strategy to
construct free groups inside interval exchange transformations would be to
show that this phenomenon has positive Lebesgue measure. Our goal in this
paper is to show that this is not the case: almost every pair of interval
exchange transformations does not generate a free group. Of course, this does
not exclude the existence of such an example, but it would have to be
non-typical both from the measure and topological points of view.

To state the theorem more precisely, let us introduce a few notations.
Consider $d \in \N^*$ and $\pi$ a permutation of $\{1, \dotsc, d\}$ and
scalars $(\lambda_1, \dotsc, \lambda_d) \in (0,1)^d$ with $\sum \lambda_i =
1$. Then one defines the associated interval exchange transformation
$f=f_{\pi, \lambda}$ from $[0,1)$ to itself as follows: partition $[0,1)$
into consecutive intervals $I_1,\dotsc, I_d$ of length $\lambda_i$, and then
swap them according to $\pi$. By convention, all our intervals are closed on
the left and open on the right, i.e., of the form $[a, b)$. We will sometimes
write $I_i = [x_{i-1}, x_i)$, i.e., the $i$-th singularity $x_i=\sum_{j< i+1}
\lambda_j$ is the left endpoint of $I_{i+1}$. See Figure~\ref{fig:iet} for an
example. When $d$ and $\pi$ are fixed, denote by $\Delta_\pi$ the set of
corresponding $(\lambda_1, \dotsc, \lambda_d)$. It is a $(d-1)$-dimensional
open simplex as $\lambda_d = 1 -\sum_{i<d} \lambda_i$, and it has a canonical
Lebesgue measure class (given for instance by $\dd\lambda_1 \wedge \dotsb
\wedge \dd\lambda_{d-1}$).

\begin{figure}
\centering
\begin{tikzpicture}[
  x=2cm,y=1cm,
  interval/.style={line width=2.5pt},
  map/.style={-{Stealth[length=2mm]},thick},
  every node/.style={font=\normalsize}
]

\coordinate (t0) at (0,0);
\coordinate (t1) at (1.35,0);
\coordinate (t2) at (2.55,0);
\coordinate (t3) at (4.65,0);
\coordinate (t4) at (5.65,0);
\coordinate (t5) at (7.0,0);

\coordinate (b0) at (0,-4.2);
\coordinate (b1) at (1.35,-4.2);
\coordinate (b2) at (2.35,-4.2);
\coordinate (b3) at (4.45,-4.2);
\coordinate (b4) at (5.65,-4.2);
\coordinate (b5) at (7.0,-4.2);

\draw[thin] (t0) -- (t5);
\draw[thin] (b0) -- (b5);

\draw[interval,blue!75!black]   (t0) -- (t1);
\draw[interval,green!50!black]  (t1) -- (t2);
\draw[interval,gray!50]         (t2) -- (t3);
\draw[interval,violet!80!black] (t3) -- (t4);
\draw[interval,red!75!black]    (t4) -- (t5);

\draw[interval,red!75!black]    (b0) -- (b1);
\draw[interval,violet!80!black] (b1) -- (b2);
\draw[interval,gray!50]         (b2) -- (b3);
\draw[interval,green!50!black]  (b3) -- (b4);
\draw[interval,blue!75!black]   (b4) -- (b5);

\foreach \p in {0,...,5} {
  \fill (t\p) circle (2pt);
  \fill (b\p) circle (2pt);
}

\node[below=3pt] at (t0) {$x_0$};
\node[below=3pt] at (t1) {$x_1$};
\node[below=3pt] at (t2) {$x_2$};
\node[below=3pt] at (t3) {$x_{d-2}$};
\node[below=3pt] at (t4) {$x_{d-1}$};
\node[below=3pt] at (t5) {$1$};

\draw[decorate,decoration={brace,amplitude=5pt}]
  (0,0.15) -- (1.35,0.15)
  node[midway,above=6pt] {$I_1$};
\draw[decorate,decoration={brace,amplitude=5pt}]
  (1.35,0.15) -- (2.55,0.15)
  node[midway,above=6pt] {$I_2$};
\node at (3.5,0.48) {$\cdots$};
\draw[decorate,decoration={brace,amplitude=5pt}]
  (4.65,0.15) -- (5.65,0.15)
  node[midway,above=6pt] {$I_{d-1}$};
\draw[decorate,decoration={brace,amplitude=5pt}]
  (5.65,0.15) -- (7.0,0.15)
  node[midway,above=6pt] {$I_d$};

\node[above=3pt] at (b0) {$f(x_{d-1})$};
\node[above=3pt] at (b1) {$f(x_{d-2})$};
\node[above=3pt] at (b3) {$f(x_1)$};
\node[above=3pt] at (b4) {$f(x_0)$};
\node[above=3pt] at (b5) {$1$};

\draw[decorate,decoration={brace,amplitude=5pt,mirror}]
  (0,-4.35) -- (1.35,-4.35)
  node[midway,below=6pt] {$f(I_d)$};
\draw[decorate,decoration={brace,amplitude=5pt,mirror}]
  (1.35,-4.35) -- (2.35,-4.35)
  node[midway,below=6pt] {$f(I_{d-1})$};
\node at (3.5,-4.67) {$\cdots$};
\draw[decorate,decoration={brace,amplitude=5pt,mirror}]
  (4.45,-4.35) -- (5.65,-4.35)
  node[midway,below=6pt] {$f(I_2)$};
\draw[decorate,decoration={brace,amplitude=5pt,mirror}]
  (5.65,-4.35) -- (7.0,-4.35)
  node[midway,below=6pt] {$f(I_1)$};

\draw[map,blue!75!black]
  (0.675,-0.08)
  .. controls (0.675,-2.0) and (6.325,-2.0) ..
  (6.325,-4.08);

\draw[map,green!50!black]
  (1.95,-0.08)
  .. controls (1.95,-1.5) and (5.05,-2.8) ..
  (5.05,-4.08);

\draw[map,violet!80!black]
  (5.15,-0.08)
  .. controls (5.15,-2.8) and (1.85,-1.5) ..
  (1.85,-4.08);

\draw[map,red!75!black]
  (6.325,-0.08)
  .. controls (6.325,-2.0) and (0.675,-2.0) ..
  (0.675,-4.08);
\end{tikzpicture}
\caption{An interval exchange transformation for the order-reversing permutation $\pi:(1,\dotsc, d) \mapsto (d, \dotsc, 1)$}
\label{fig:iet}
\end{figure}
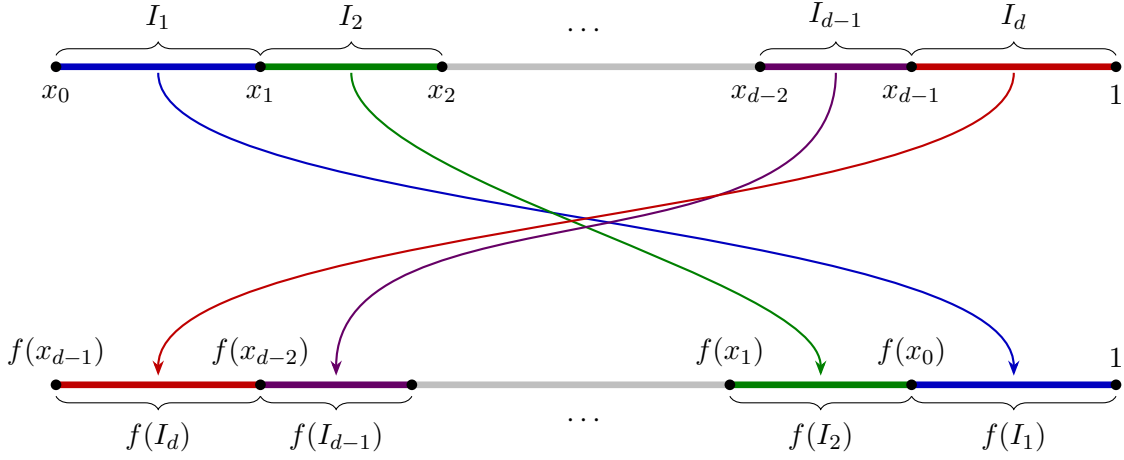

\begin{thm}
\label{main_thm} Let $d\in \N^*$ and $\pi$ a permutation of $\{1, \dotsc,
d\}$. Let $e\in \N^*$ and $\sigma$ a permutation of $\{1, \dotsc, e\}$. Then
there is a full Lebesgue measure subset $\Omega$ of $\Delta_\pi \times
\Delta_{\sigma}$ such that, for $(\lambda, \mu)\in \Omega$, the two
transformations $f_{\pi, \lambda}$ and $f_{\sigma, \mu}$ do not generate a
free group: there is a nontrivial word $w = w(x,y)$ such that $w(f_{\pi,
\lambda}, f_{\sigma, \mu}) = \id$.
\end{thm}
\begin{rmk}
The translation vectors associated to compositions of $f_{\pi, \lambda}$,
$f_{\sigma, \mu}$ and their inverses are of the form $k + \sum_{i=1}^{d-1}
a_i \lambda_i + \sum_{i=1}^{e-1} b_i \mu_i$ for some integers $k, a_1,\dotsc,
a_{d-1}, b_1, \dotsc, b_{e-1}$. There is a full measure set $\Omega'$ such
that, for $(\lambda, \mu) \in \Omega'$, such a linear combination vanishes
only when all the integers are zero, as this amounts to excluding countably
many affine hyperplanes. It follows that, in $\Omega'$, if a composition of
$f_{\pi, \lambda}$, $f_{\sigma, \mu}$ and their inverses sends a singularity
to a singularity, then the integers that show up in the relation are all
zero. Therefore, this coincidence will still be valid for small perturbations
of $f_{\pi, \lambda}$ and $f_{\sigma, \mu}$.

Consider $(\lambda,\mu) \in \Omega \cap \Omega'$, where $\Omega$ is the set
constructed by Theorem~\ref{main_thm}, and the relation $w$ given by this
theorem. Then this relation is robust: small perturbations of $f_{\pi,
\lambda}$ and $f_{\sigma, \mu}$ satisfy the same relation, as the pattern of
singularities and their images remain the same for the original maps and
their small perturbations.

In particular, Theorem~\ref{main_thm} automatically gives a full measure
\emph{open} set of pairs of interval exchange transformations not generating
a free group. This
recovers~\cite[Theorem~5.2]{dahmani_fujiwara_guirardel_iet_free} with a
different construction (and weakening very slightly its assumptions as we do
not need $\sigma$ to be admissible in their sense).
\end{rmk}

The strategy of the proof of Theorem~\ref{main_thm} is starting from the idea
of~\cite{dahmani_fujiwara_guirardel_iet_free}: find an element that is a
small translation on a large set, form a commutator to obtain an element with
small support, then try to move things around to make the supports disjoint
and take again commutators to get the identity. The moving around argument
will fail and instead we will iterate commutators to get elements whose
support has exponentially small measure. To illustrate the argument in the
simplest case, where one does not need to iterate commutators, let us sketch
the proof of the following easy fact:
\begin{lem}
\label{lem:rot} Let $R$ be the rotation by some irrational angle $\theta$,
and let $f$ be an interval exchange transformation. Then $f$ and $R$ do not
generate a free group.
\end{lem}
\begin{proof}
Consider $n$ such that $t_n=n\theta$ modulo $1$ is close to $0$. Then $R^n$
is a small translation by $t_n$. Outside of a small neighborhood of the
singularities of $f$, then $h \coloneqq [f, R^n] \coloneqq f^{-1} R^{-n} f
R^n$ is the identity as it maps $x$ to $x + t_n + v - t_n - v$ for some
vector $v$. Consider $q$ such that $R^q$ makes the singularity set of $f$
disjoint from itself. Then $R^{q} h R^{-q}$ has a support close to this
shifted singularity set, so disjoint from the support of $h$. It follows that
$[h, R^{q} h R^{-q}]= \id$, which is the desired relation.
\end{proof}

We will tweak this basic strategy to prove Theorem~\ref{main_thm}, as
follows.

\medskip

Assume first that the permutation $\pi$ is irreducible, i.e., there is no
$0<k < d$ such that $\pi(\{1, \dotsc, k\}) = \{1,\dotsc, k\}$. Then it is a
standard consequence of Rauzy-Veech induction that, if $\lambda$ is typical,
there is an iterate $f_{\pi,\lambda}^q$ which equals a small translation on a
set of measure very close to $1$. It follows that $h_0 \coloneqq [f_{\sigma,
\mu}, f_{\pi,\lambda}^q]$ is the identity on a set of measure at least
$1-\epsilon$, as in the proof of Lemma~\ref{lem:rot}. However, we have no
explicit control on the geometric shape of the support of $h_0$, so there is
no direct way to translate things as in the proof of the lemma to get
disjoint supports. It turns out that we will not need to understand this
shape. Instead, we will iterate commutators (with suitable conjugacies to
make things quantitatively disjoint, thanks to Birkhoff theorem), to obtain a
sequence $h_n$ whose support has a measure tending to $0$ very quickly, like
$\epsilon^n$. Moreover, the word length of $h_n$ will grow at most like $C
2^n$.

We will get a contradiction from a diophantine argument. If $h$ belongs to
the group generated by $f_{\pi,\lambda}$ and $f_{\sigma, \mu}$, then its
singularities are located at points of the form $\sum_{i=1}^{d} a_i \lambda_i
+ \sum_{j=1}^{e} b_j \mu_j$ for some integers $a_1,\dotsc, a_d, b_1, \dotsc,
b_e$. The difference between two successive singularities has the same shape.
When such an expression is nonzero, a Borel-Cantelli argument shows that for
typical $(\lambda, \mu)$ it is bounded below by $\max(\max_i \abs{a_i},
\max_j \abs{b_j})^{-c}$ for some~$c$. Therefore, this quantity is a lower
bound on the measure of the support of $h$ if it is not the identity. When
$h$ has word length $N$, then all the integers in this formula are bounded by
$C N$, and we get a lower bound for the measure of the support of the form
$1/N^c$.

Applying this argument to $h_n$ gives a contradiction: if it is not the
identity, then the measure of its support should both be smaller than
$\epsilon^n$ and larger than $1/(2^n)^c$, which is absurd if $\epsilon$ is
small enough to start with.

\medskip

If $\pi$ is not irreducible, this argument does not go through. However, we
can decompose maximally $[0,1)$ into the disjoint union of some intervals
$P_i$ stable under both $f_{\pi,\lambda}$ and $f_{\sigma, \mu}$. On each
$P_i$, the composition $f_{\sigma, \mu} f_{\pi,\lambda}$ is an interval
exchange transformation for an irreducible permutation, and its interval
lengths have a distribution which is absolutely continuous with respect to
Lebesgue measure. Therefore, we may apply the irreducible case to the pair
$(f_{\sigma, \mu}f_{\pi,\lambda}, f_{\sigma, \mu})$ to get a relation $w_i$
that they satisfy on $P_i$. Taking commutators of the $w_i$, one gets a
nontrivial relation valid on the whole interval $[0,1)$.

Note that the pair $(f_{\sigma, \mu} f_{\pi,\lambda}, f_{\sigma, \mu})$ is
not an independent pair, so Theorem~\ref{main_thm} in the irreducible case is
not enough to conclude in the non-irreducible case. However, the proof in the
irreducible case does not need the full strength of independence: we will
give a more precise statement in Theorem~\ref{thm:irred}, requiring only that
the first map is uniquely ergodic and rigid by translations and the pair is
diophantine, that will be sufficient for the reduction.

\medskip

Section~\ref{sec:irred} gives the details of the above argument when $\pi$ is
irreducible, culminating in Theorem~\ref{thm:irred}. Section~\ref{sec:red}
completes the proof of Theorem~\ref{main_thm} by reducing it to the
irreducible case.

\medskip

\textbf{AI methodology.} I had been thinking about this problem some years
ago, trying to take advantage of the fact that Rauzy-Veech induction makes it
possible to create commutators with small support. I was stuck due to the
lack of a good enough geometric control on the shape of this support. I gave
one single prompt to ChatGPT 6 Astra with the general strategy of the proof
and a record of my failed attempts. It answered back the (clever and nice!)
idea to iterate the commutators and only rely on a measure control instead of
a geometric control. I digested, simplified and extended its argument. The
paper is entirely human-written.

\medskip

\textbf{Acknowledgments.} This paper grew out of a collective experiment in
the ergodic theory team at Rennes University around AI tools, to discuss
knowingly how we wanted to use them or to avoid them.

\section{The irreducible case}

\label{sec:irred}

In this section, we give a precise version of Theorem~\ref{main_thm} when the
permutation $\pi$ is irreducible. We will use explicit conditions for the
interval exchange transformations, formulated in the following definitions.

\begin{definition}
An interval exchange transformation is uniquely ergodic if its only invariant
probability measure is Lebesgue measure.
\end{definition}
Almost every interval exchange transformation over $d\ge 2$ intervals is
uniquely ergodic, see e.g.~\cite[Theorem~1.2]{veech_metric_1}.

\begin{definition}
An interval exchange transformation $f$ is rigid by translations if there
exist $q_n \to \infty$ and $t_n \to 0$ and measurable sets $B_n$ with
$\Leb(B_n) \to 1$ such that, for all $n$ and all $x \in B_n$,
\begin{equation*}
  f^{q_n}(x) = x + t_n.
\end{equation*}
\end{definition}

\begin{definition}
A pair of interval exchange transformations $(f_{\pi, \lambda}, f_{\sigma,
\mu})$ is diophantine if there exist $C>0$ and $c>0$ such that, for all
integers $a_1,\dotsc, a_{d}, b_1,\dotsc, b_{e} \in \Z$,
\begin{equation*}
  \abs*{\sum_{i=1}^d a_i \lambda_i + \sum_{j=1}^e b_j \mu_j} \geq C / \max (\max_i \abs{a_i}, \max_j \abs{b_j})^c
\end{equation*}
whenever the left hand term is nonzero.
\end{definition}

These conditions have full Lebesgue measure, as we explain now.

\begin{prop}
\label{prop:rigid_transl} Assume that $\pi$ is irreducible over $d\ge 2$
intervals. Then, for Lebesgue almost every $\lambda \in \Delta_\pi$, the
interval exchange transformation $f_{\pi,\lambda}$ is rigid by translations.
\end{prop}
\begin{proof}
This is a standard consequence of the ergodicity of Rauzy-Veech induction,
more precisely of the fact that for some induced map, one interval will be
much larger than the other ones, and show up with frequency bounded from
below at previous steps. For non-experts, let us derive the result
from~\cite[Theorem~1.4]{veech_metric_1} which is proved by implementing this
strategy.

We apply Veech's theorem to $\epsilon=1/n$. For almost every $f=f_{\pi,
\lambda}$, we get a large $N$ (this will be our $q_n$) and an interval $J$
whose iterates until time $N$ are disjoint and singularity-free, and cover a
proportion $1-\epsilon$ of the space, and the interval $K \coloneqq f^{-N} J
\cap J$ covers a proportion at least $1-\epsilon$ of $J$. Take $B_n =
\bigcup_{i=0}^{N-1} f^i (K)$. Its measure is close to $1$. As $N$ is large,
the interval $J$ is short. Since $f^{-N} J \cap J \ne \emptyset$, this
implies that $f^N$ is a small translation by a number $t_n \to 0$ on $J$, and
therefore on $K$. Let $i<N$. Applying $f^i$ to two points $x \in K$ and $f^N
x = x+t_n \in J$, one deduces that $f^N$ is the translation by $t_n$ on
$f^i(K)$.
\end{proof}

\begin{prop}
\label{prop:dioph} Fix two permutations $\pi$ and $\sigma$. Then almost every
pair of interval exchange transformations in $\Delta_\pi \times
\Delta_\sigma$ is diophantine.
\end{prop}
\begin{proof}
Define $\psi(q) = 1/q^c$. It is enough to show that, almost surely, there are
only finitely many solutions in integer numbers $(k, a_1,\dotsc, a_{d-1},
b_1,\dotsc, b_{e-1})$ to the inequality
\begin{equation*}
  \abs*{k + \sum_{i=1}^{d-1} a_i \lambda_i + \sum_{i=1}^{e-1} b_i \mu_i} < \psi\pare*{\sum \abs{a_i} + \sum \abs{b_i}}.
\end{equation*}
Indeed, note that $\lambda_d = 1 - \sum_{i=1}^{d-1} \lambda_i$ and $\mu_e = 1
- \sum_{i=1}^{e-1} \mu_j$, so the contributions of $\lambda_d$ and
$\lambda_e$ can be moved to a common integer $k$.

The easy direction of the Khintchine-Groshev theorem (see
e.g.~\cite{beresnevich_velani}) shows that this is the case when
\begin{equation}
\label{eq:sum_groshev}
  \sum q^{(d-1 + e-1) - 1} \psi(q) < \infty.
\end{equation}
The result follows for $c=d+e-1$ for instance.

Note that there is nothing deep in this proof: it is a straightforward
Borel-Cantelli argument. The solutions to the inequality form a slightly
thickened hyperplane, and the sum of their measures over varying $(k, a_1,
\dotsc, a_{d-1}, b_1, \dotsc, b_{e-1})$ is finite when~\eqref{eq:sum_groshev}
holds.
\end{proof}

Here is the precise version of Theorem~\ref{main_thm} in the irreducible
case.

\begin{thm}
\label{thm:irred} Let $\pi$ be a permutation of $\{1,\dotsc, d\}$, and let
$\sigma$ be a permutation of $\{1,\dotsc, e\}$. Consider $\lambda \in
\Delta_\pi$ and $\mu \in \Delta_\sigma$ such that $f_{\pi,\lambda}$ is
uniquely ergodic and rigid by translations, and $(f_{\pi, \lambda},
f_{\sigma, \mu})$ is diophantine. Then the pair $(f_{\pi, \lambda},
f_{\sigma, \mu})$ does not generate a free group: there exists a nontrivial
word $w \in F_2(x,y)$ with $w(f_{\pi, \lambda}, f_{\sigma, \mu}) = \id$.
\end{thm}
When $\pi$ is irreducible and $d\ge 2$, the three conditions have full
measure, by the propositions above. Therefore, this theorem implies
Theorem~\ref{main_thm} in the irreducible case. The additional information we
get by specifying the shape of the full measure set will be instrumental in
the reduction from the general case to the irreducible case.

Before starting the proof, let us record an easy lemma which will be used to
ensure that the words we construct are always nontrivial.
\begin{lem}
\label{lem:exists_nontrivial_conj} Let $u$ and $v$ be two nontrivial words in
$F_2(x,y)$. Then $[u, v] \ne 1$ or $[x^{-1}ux, v] \ne 1$ or $[y^{-1}uy, v]
\ne 1$.
\end{lem}
\begin{proof}
Assume by contradiction that all three commutators are the identity. The
three words $u$, $x^{-1}ux$ and $y^{-1}uy$ belong to the centralizer of $v$,
which is cyclic by~\cite[Proposition~2.19]{lyndon_schupp}. As $x^{-1}ux$ and
$u$ have the same reduced length, it follows that $x^{-1}ux = u^{\pm 1}$. But
$u$ is not conjugate to $u^{-1}$ by the proof of
~\cite[Proposition~2.19]{lyndon_schupp}, so $x^{-1}ux=u$. In particular, $x$
belongs to the centralizer of $u$. By the same argument, $y$ also belongs to
this centralizer. This is a contradiction as this centralizer is cyclic.
\end{proof}

The rest of this section is devoted to the proof of Theorem~\ref{thm:irred}.
From this point on, we fix $\lambda$ and $\mu$ satisfying its assumptions. We
will write $f = f_{\pi, \lambda}$ and $g = f_{\sigma, \mu}$. Let also
$\epsilon>0$ be small enough (how small will be determined at the end of the
proof). We will write $\supp h$ for the support of an interval exchange
transformation $h$, i.e., the points which are not fixed by $h$. Let us also
denote by $\abs{w}$ the word length of a word $w \in F_2(x,y)$.

\begin{lem}
\label{lem:small_support} There exists a nontrivial word $w_0 \in F_2(x,y)$
such that the support of $w_0(f, g)$ has measure at most $\epsilon$.
\end{lem}
\begin{proof}
As $f$ is rigid by translations, we may find $q$ large and $t$ small and $B$
of measure at least $1-\epsilon/3$ such that $f^q (x) = x + t$ on $B$. If $t$
is small enough, the $t$-neighborhood $U$ of the singularities of $g$ has
measure at most $\epsilon/3$.

Outside of $B^c \cup U \cup g^{-1} B^c$, we claim that $[g, f^q](x) = x$.
Indeed, $x$ and $f^q x = x+t$ are in the same linearity interval of $g$ as $x
\notin U$, with a translation vector $v$. Then $g f^q (x) = g (x + t) =
x+t+v$, while $f^q g (x) = f^q(x + v) = x + v + t$ as $g(x)\in B$. These two
quantities coincide, proving the claim.

It follows that the measure of the support of $[g, f^q]$ is at most $3 \cdot
\epsilon/3 = \epsilon$.
\end{proof}

Let $S_0$ be the support of $h_0 \coloneqq w_0(f, g)$ as constructed in
Lemma~\ref{lem:small_support}. It is a finite union of intervals, and
satisfies $\Leb(S_0) \le \epsilon$. We will now use unique ergodicity to make
sure that most points do not visit $S_0$ often.

\begin{lem}
There exists $N$ such that, for any $x \in [0,1)$,
\begin{equation}
\label{eq:uniform}
  \frac{1}{N} \sum_{i=0}^{N-1} 1_{S_0}(f^i x) \le 2\epsilon.
\end{equation}
\end{lem}
\begin{proof}
If $f$ and $1_{S_0}$ were continuous, then unique ergodicity of $f$ would
ensure uniform convergence of the Birkhoff sums of $1_{S_0}$ to its average,
which is $\le \epsilon$ by Lemma~\ref{lem:small_support}. Therefore, the
conclusion would follow.

Unfortunately, the objects we are considering are not continuous. This is
only a minor annoyance, though: one may blow up the singularities of $f$ and
the endpoints of $S_0$, together with their whole orbits, to recover a
uniquely ergodic homeomorphism of a compact space and a continuous function,
and apply the argument there. See for instance~\cite[Page~55]{yoccoz_iet}.
\end{proof}

Let us fix $N$ as given by the previous lemma. Taking commutators with
suitable conjugates of $h_0$ and iterating this process, we will be able to
get elements with exponentially small support. Let us first describe one
single step, done with a single commutator.

\begin{lem}
\label{lem:iterate} Let $w \in F_2(x,y)$ be a nontrivial word. Then there
exists a nontrivial word $w' \in F_2(x,y)$, with length at most $2 \abs{w} +
4N + 2\abs{w_0}$, such that
\begin{equation*}
  \Leb(\supp w'(f, g)) \leq 18 \epsilon \Leb (\supp w(f, g)).
\end{equation*}
\end{lem}
\begin{proof}
Let us integrate the inequality~\eqref{eq:uniform} over the support $S$ of
$w(f, g)$. We get the inequality
\begin{equation*}
  \frac{1}{N} \sum_{i=0}^{N-1} \Leb(S \cap f^{-i} S_0) \leq 2 \epsilon \Leb(S).
\end{equation*}
We may also apply the same argument to $f^{-1} S$, and to $g^{-1} S$ (which
both have the same measure as $S$). Summing the resulting inequalities, we
get
\begin{equation*}
  \frac{1}{N} \sum_{i=0}^{N-1} (\Leb(S \cap f^{-i} S_0) + \Leb(f^{-1} S \cap f^{-i} S_0) + \Leb(g^{-1} S \cap f^{-i} S_0)) \leq 6 \epsilon \Leb(S).
\end{equation*}
In particular, we can fix $i<N$ such that
\begin{equation}
\label{eq:borne_somme}
  \Leb(S \cap f^{-i} S_0) + \Leb(f^{-1} S \cap f^{-i} S_0) + \Leb(g^{-1} S \cap f^{-i} S_0) \leq 6 \epsilon \Leb(S).
\end{equation}

We note that, for any injective transformations $\phi, \psi$, writing $A =
\supp \phi \cap \supp \psi$, we have
\begin{equation*}
  \supp[\phi, \psi] \subseteq A \cup \phi^{-1} A \cup \psi^{-1} A.
\end{equation*}
Indeed, if $x$ does not belong to the set on the right, one checks readily
that $\phi \psi(x) = \psi\phi(x)$ by separating the cases where $\phi(x) \ne
x$ (in which case $\psi(x) = x$ and $\psi(\phi(x)) = \phi(x)$) and the
symmetric case $\psi(x)\ne x$. In particular, when $\phi$ and $\psi$ are
interval exchange transformations, which preserve Lebesgue measure, we get
\begin{equation*}
  \Leb(\supp [\phi, \psi]) \leq 3 \Leb(\supp \phi \cap \supp \psi).
\end{equation*}
Let us apply this inequality to $\phi = w(f,g)$ (with support $S$) and $\psi
= f^{-i} h_0 f^i$ (with support $f^{-i} S_0$). We get
\begin{equation*}
  \Leb(\supp [w(f,g), f^{-i} h_0 f^i]) \leq 3 \Leb(S \cap f^{-i} S_0) \leq 18\epsilon \Leb(S),
\end{equation*}
thanks to~\eqref{eq:borne_somme}. If the word $w' = [w, x^{-i} w_0 x^i]$ is
nontrivial, this proves the proposition.

If this word is trivial, we may apply the same argument to $\phi = f^{-1}
w(f, g) f$ (with support $f^{-1} S$) and $\psi = f^{-i} h_0 f^i$ (with
support $f^{-i} S_0$). As the measure of $f^{-1} S \cap f^{-i} S_0$ is also
controlled by~\eqref{eq:borne_somme}, we also get the inequality
\begin{equation*}
  \Leb(\supp [f^{-1} w(f,g) f, f^{-i} h_0 f^i]) \leq 18\epsilon \Leb(S).
\end{equation*}
Finally, this inequality is also true for $[g^{-1} w(f,g) g, f^{-i} h_0 f^i]$
by the same argument.

By Lemma~\ref{lem:exists_nontrivial_conj}, one of the words $[w, x^{-i} w_0
x^i]$ and $[x^{-1}wx, x^{-i} w_0 x^i]$ and $[y^{-1}wy, x^{-i} w_0 x^i]$ is
nontrivial. This word satisfies the conclusion of the lemma.
\end{proof}

Starting from $w_0$ and applying Lemma~\ref{lem:iterate} iteratively, we get
a sequence of nontrivial words $w_n$ with
\begin{equation}
\label{eq:abs_wnp1}
  \abs{w_{n+1}} \le 2\abs{w_n} + K
\end{equation}
for $K=4N + 2\abs{w_0}$, and with
\begin{equation}
\label{eq:leb_supp}
  \Leb(\supp w_n(f, g)) \leq \epsilon (18\epsilon)^n.
\end{equation}
The equation~\eqref{eq:abs_wnp1} gives $\abs{w_{n+1}} + K \le 2 (\abs{w_n} +
K)$, and therefore
\begin{equation}
\label{eq:length_wn}
  \abs{w_n} \le 2^n (\abs{w_0} + K).
\end{equation}

To conclude the proof of Theorem~\ref{thm:irred}, it suffices to show that
$h_n \coloneqq w_n(f, g)$ acts as the identity for large enough $n$. Assume
that this is not the case. Then there is a nontrivial interval $I_n$ on which
$h_n$ acts as a nontrivial translation, with singularities at its endpoints.
Let $\ell_n$ be its length. Then $h_n$ is not the identity on $I_n$, so its
support has measure at least $\ell_n$. Together with~\eqref{eq:leb_supp},
this gives the inequality
\begin{equation}
\label{eq:ell_le}
  \ell_n \le \epsilon (18\epsilon)^n.
\end{equation}

The translation vector of $f$ for a point $x$ in the $i$-th interval is given
by
\begin{equation}
\label{eq:translation_vector}
  \sum_{\pi(j) < \pi (i)} \lambda_j - \sum_{j<i} \lambda_j.
\end{equation}
See e.g.~\cite[Page~8]{yoccoz_iet}. A similar expression holds for $f^{-1}$
and $g$ and $g^{-1}$. It follows that all singularities of $h_n$ are of the
form $\sum_{i=1}^{d} a_i \lambda_i + \sum_{j=1}^{e} b_j \mu_j$, as a point
which is a singularity has to be mapped by a sum of such translation vectors
to some singularity of the original maps. The length $\ell_n$, which is the
difference of the endpoints of $I_n$, is also of this form, and it is nonzero
as $I_n$ is nontrivial. Write $\ell_n = \sum_{i=1}^{d} a_i \lambda_i +
\sum_{j=1}^{e} b_j \mu_j$. The diophantine condition for the pair $(f, g)$
ensures that $\ell_n$ is bounded below by $C / \max(\max_i \abs{a_i}, \max_j
\abs{b_j})^c$. The numbers $a_i$ and $b_j$ that show up in this formula are
all bounded linearly in terms of the word length of $\abs{w_n}$, as the
translation vectors for one map involve coefficients bounded by $1$ in
absolute value, thanks to~\eqref{eq:translation_vector}. For some constant
$C'>0$, we get the inequality
\begin{equation*}
  \frac{C'}{\abs{w_n}^c} \leq \ell_n.
\end{equation*}
The quantity $\abs{w_n}$ grows at most like $2^n$ by~\eqref{eq:length_wn}.
Therefore, we get a contradiction with~\eqref{eq:ell_le} for large $n$ if
$18\epsilon<1/2^c$. This can be ensured by taking $\epsilon$ small enough at
the beginning of the argument. This completes the proof of
Theorem~\ref{thm:irred}.

\section{The general case}

\label{sec:red}

In this section, we prove Theorem~\ref{main_thm}. Let us fix $d,e\in \N^*$
and $\pi, \sigma$ two permutations respectively of $\{1,\dotsc, d\}$ and
$\{1, \dotsc, e\}$. Let $(\lambda,\mu) \in \Delta_\pi \times \Delta_\sigma$.

There are finitely many expressions of the form $k + \sum_{i=1}^{d-1} a_i
\lambda_i + \sum_{j=1}^{e-1} b_i \mu_i$ in which the integers $k, a_1,
\dotsc, a_{d-1}, b_1, \dotsc, b_{d-1}$ are all bounded by $4$ in absolute
value. All these quantities are nonzero unless all the integers are $0$,
except in finitely many $(\lambda,\mu)$ affine hyperplanes with zero measure.
Let $O$ be a connected component of the complement of these hyperplanes: this
fixes the relative positions of the singularities and of their first
iterates. It suffices to prove Theorem~\ref{main_thm} for almost every
$(\lambda,\mu) \in O$. We will write $f=f_{\pi,\lambda}$ and $g=f_{\sigma,
\mu}$.

The map $gf$ is an interval exchange transformation on $[0,1)$, with respect
to a permutation $\rho$ which does not depend on the parameters in $O$. Its
singularity set is equal to $S_f \cup f^{-1} S_g$, where $S_f$ and $S_g$
denote the singularity sets of $f$ and $g$ (not including $0$). These
singularities are all distinct, and there are $d-1 + e-1$ of them. Therefore,
$\rho$ is a permutation of $\{1,\dotsc, d+e-1\}$, and the corresponding
parameters for $gf$ live in a simplex of dimension $d+e-2$.

\begin{lem}
\label{lem:ac} The map $L_O : O \subseteq \Delta_\pi \times \Delta_\sigma \to
\Delta_\rho$, associating to $(\lambda, \mu)$ the interval lengths of the
interval exchange transformation $gf$, is an invertible affine map.
\end{lem}
\begin{proof}
The singularities of $gf$ are located at points of the form $\sum_{i<n}
\lambda_i$ (for singularities of $f$) or $\sum_{j<m} \mu_j -
(\sum_{\pi(i)<\pi(n)} \lambda_i - \sum_{i<n} \lambda_i)$ (for singularities
of $g$ minus a translation vector of $f$). These expressions are affine in
$(\lambda_1,\dotsc, \lambda_{d-1}, \mu_1, \dotsc, \mu_{e-1})$. So are the
differences of consecutive singularities, i.e., the lengths of the intervals
of $gf$. This shows that $L_O$ is affine.

To show that it is invertible, we build an inverse map. As the singularities
of $f$ are also singularities of $gf$, the linearity intervals for $f$ are
unions of linearity intervals of $gf$, therefore we may recover the
$\lambda_i$ as sums of interval lengths of $gf$. In the same way, if $J$ is a
linearity interval of $g$, then $f^{-1} J$ is a union of linearity intervals
of $gf$, hence its measure (which equals the length of $J$) can be recovered
as a sum of lengths of intervals for $gf$.
\end{proof}

Let us say that the $k$-th singularity of $f$ is \emph{splitting} if
$\pi\{1,\dotsc, k\} = \{1,\dotsc, k\}$. This is the obstruction to the
irreducibility of $\pi$. As singularities of $f$ and $g$ are disjoint in $O$,
there can be no common splitting singularity. However, there is still a
mechanism by which the action of the group generated by $f$ and $g$ is not
transitive. Let us say that the $\ell$-th interval $J_\ell$ for $g$ is
\emph{fixed} if both the $\ell-1$-th and $\ell$-th singularities are
splitting, i.e., $\sigma\{1,\dotsc, \ell-1\} = \{1,\dotsc, \ell-1\}$ and
$\sigma(\ell) = \ell$. Then $g$ acts as the identity on $J_\ell$. The
admissibility condition in~\cite{dahmani_fujiwara_guirardel_iet_free,
jay_iet} requires precisely that there is no fixed interval. If a splitting
singularity $x$ of $f$ belongs to a fixed interval for $g$, then the
intervals $[0,x)$ and $[x, 1)$ are invariant under both $f$ and $g$. We
define the set of \emph{pair-splitting singularities} for $(f, g)$ as the set
of splitting singularities of $f$ which belong to a fixed interval for $g$,
or conversely. We also say that $0$ and $1$ are pair-splitting singularities
by convention.

The key point for the reduction to the irreducible case is that, on intervals
between successive pair-splitting singularities, the dynamics is irreducible.
More precisely, the interval exchange transformation $gf$ is associated to an
irreducible permutation there.
\begin{lem}
\label{lem:irred_pair} Let $(\lambda,\mu)\in O$. Let $u < v$ be two
consecutive pair-splitting singularities for $(f, g)$. Then the interval
$P=[u,v)$ is invariant under both $f$ and $g$. Moreover, the restriction to
$P$ of the permutation $\rho$ associated to the interval exchange
transformation $gf$ is irreducible.
\end{lem}

The proof of this lemma is mainly a combinatorics exercise. We record the
following easy fact on permutations that we will use several times.
\begin{lem}
\label{lem:perm} Let $\rho$ be a permutation of $\{1, \dotsc, d\}$ for some
$d$, and let $a,b\in \{1,\dotsc, d\}$. Assume that $\{i \st \rho(i) < \rho
(a)\} = \{i \st i < b\}$. Then $\rho(a) = b$ and $\rho\{1, \dotsc, b-1\} =
\{1, \dotsc, b-1\}$.
\end{lem}
\begin{proof}
The equality $\rho(a) = b$ follows by taking the cardinalities in the
equality $\{i \st \rho(i) < \rho (a)\} = \{i \st i < b\}$. This equality then
reads $\rho(i)< b \Leftrightarrow i < b$, which gives $\rho\{1, \dotsc, b-1\}
= \{1, \dotsc, b-1\}$.
\end{proof}

\begin{proof}[Proof of Lemma~\ref{lem:irred_pair}]
The invariance under both $f$ and $g$ is obvious from the definition. The
content of the lemma is the irreducibility of the restriction of $\rho$ to
$P$.

To simplify notations, we will assume that $u=0$ and $v<1$. The other cases
$P=[u,1)$ with $u>0$, and $P=[0,1)$, and $P=[u,v)$ with $u>0$ and $v<1$, are
analogous.

Let $S_f$ be the set of singularities of $f$ in $P$, and $S_g$ those of $g$.
The singularities of $gf$ in $P$ are $S \coloneqq S_f \cup f^{-1} S_g$.

Let us show that the restriction to $P$ of the permutation $\rho$ associated
to $gf$ is irreducible. Otherwise, considering $m>0$ such that
$\rho\{1,\dotsc, m\} = \{1,\dotsc, m\}$, the right end-point of the
$m$\nobreakdash-th interval is positive and it is the image under $gf$ of a
singularity $x\in P$. This entails a nontrivial integer relation between
$(\lambda_i)$ and $(\mu_j)$, which is not possible in $O$. Let us give more
details, separating the 4 cases $x \in S_f$ or $f^{-1} S_g$ and $gfx \in S_f$
or $f^{-1} S_g$.

\emph{Case 1: Assume first $x \in S_f$}. Denoting by $k$ the index of the
singularity, then $x = \sum_{i<k+1} \lambda_i$. Then $f x = \sum_{\pi(i) <
\pi(k+1)} \lambda_i$. It belongs to some linearity interval for $g$, say with
index $\ell$. Then
\begin{equation}
\label{eq:gfx}
  g f x = \sum_{\pi(i) < \pi(k+1)} \lambda_i + \sum_{\sigma(j) < \sigma(\ell)} \mu_j - \sum_{j<\ell} \mu_j.
\end{equation}

\emph{Case 1a: $gfx \in S_f$}. Say it is the $n$-th singularity. We get
\begin{equation*}
  \sum_{\pi(i) < \pi(k+1)} \lambda_i + \sum_{\sigma(j) < \sigma(\ell)} \mu_j - \sum_{j<\ell} \mu_j = \sum_{i<n+1} \lambda_i.
\end{equation*}
As $(\lambda,\mu) \in O$, this relation is only possible if each $\lambda_i$
and $\mu_j$ featuring in the formula has a coefficient $0$. In particular,
$\{j \st \sigma(j)<\sigma(\ell)\} = \{j \st j < \ell\}$. By
Lemma~\ref{lem:perm}, this implies $\sigma(\ell) = \ell$ and $\sigma
\{1,\dotsc, \ell-1\} = \{1,\dotsc, \ell-1\}$. In the same way, $\{i \st
\pi(i) < \pi(k+1)\} = \{i \st i < n+1\}$, yielding $\pi(k+1)=n+1$ and
$\pi\{1,\dotsc, n\} = \{1,\dotsc, n\}$. The point $f x = \sum_{i<n+1}
\lambda_i$ is therefore a splitting singularity of $f$ which belongs to a
fixed interval for $g$. This contradicts the fact that there is no
pair-splitting singularity in $P$.

\emph{Case 1b: $gfx \in f^{-1} S_g$}. Applying $f$, via a translation vector
$\sum_{\pi(i)< \pi(m)} \lambda_i - \sum_{i<m} \lambda_i$, sends $gfx$ to a
singularity of $g$ of the form $\sum_{j<n} \mu_j$. If $f(gfx)=0$, then $gfx$
is already a singularity of $f$, and this case is handled above. Therefore,
we can assume $fgfx \ne 0$, i.e., $n>1$. With~\eqref{eq:gfx}, we get the
equality
\begin{equation*}
  \sum_{\pi(i) < \pi(k+1)} \lambda_i + \sum_{\sigma(j) < \sigma(\ell)} \mu_j - \sum_{j<\ell} \mu_j
  + \sum_{\pi(i)< \pi(m)} \lambda_i - \sum_{i<m} \lambda_i = \sum_{j<n} \mu_j.
\end{equation*}
All $\lambda_i$ and $\mu_j$ in this formula should have a coefficient $0$. If
$\ell>1$, we get a contradiction by considering the coefficient of $\mu_1$,
as $n>1$. Therefore, $\ell=1$. We get $\sum_{\sigma(j) < \sigma(\ell)} \mu_j
= \sum_{j<n} \mu_j$. Hence, $\{j \st \sigma(j) < \sigma(\ell)\} = \{j \st j <
n\}$. By Lemma~\ref{lem:perm}, this gives $\sigma(\ell)=n$, and moreover
$\sigma\{1, \dotsc, n-1\} = \{1, \dotsc, n-1\}$, i.e., $f gf x$ is a
splitting singularity for $g$. In the same way, if $\pi(k+1)>1$ and
$\pi(m)>1$, then $\lambda_{\pi^{-1}(1)}$ appears with a strictly positive
coefficient in the above equation, a contradiction. Therefore, $\pi(k+1)=1$
or $\pi(m)=1$.

In the first case $\pi(k+1)=1$, then $\{i \st \pi(i)<\pi(m)\} = \{i \st i <
m\}$, which gives $\pi\{1,\dotsc, m-1\} = \{1,\dotsc, m-1\}$ and $\pi(m)=m$
by Lemma~\ref{lem:perm}. This shows that $gfx$ is in a fixed interval for
$f$. Moreover, $gfx=fgfx$ is a splitting singularity of $g$. In particular,
$gfx$ is a pair-splitting singularity for $(f, g)$, a contradiction.

In the second case $\pi(m)=1$, then $\{i\st \pi(i) < \pi(k+1)\} = \{i \st i <
m\}$. Therefore $\pi(k+1)=m$ and $\pi\{1,\dotsc, m-1\} = \{1,\dotsc, m-1\}$
by Lemma~\ref{lem:perm}. If $m>1$, we get that the preimage of $1$ under
$\pi$ is $<m$, which contradicts the equality $\pi(m)=1$. Therefore,
$m=1=\pi(m)$. As $\pi\{1,\dotsc, m-1\} = \{1,\dotsc, m-1\}$, we conclude as
above that $gfx$ is in a fixed interval for $f$, and therefore a
pair-splitting singularity for $(f,g)$.

\medskip

\emph{Case 2: Assume now $x \in f^{-1} S_g$}. Then $fx$ is a singularity of
$g$, and then its image $gfx$ has the form $\sum_{\sigma(j)<\sigma(\ell+1)}
\mu_j$.

\emph{Case 2a: $gfx \in S_f$}. For some $n$, we get
$\sum_{\sigma(j)<\sigma(\ell+1)} \mu_j = \sum_{i<n} \lambda_i$. Since we are
in $O$, this implies that $n=\sigma(\ell+1)=1$, and $gfx=0$. This contradicts
the assumption that $gfx$ is a non-zero singularity.

\emph{Case 2b: $gfx \in f^{-1}S_g$}. Then $fg$ maps the singularity $fx\in
S_g$ to $fgfx \in S_g$. We have already shown that this can not happen in
Case 1a, exchanging the roles of $f$ and $g$.
\end{proof}

We note that the pattern of pair-splitting singularities is constant over
$O$, as well as the permutation $\rho$ associated to the interval exchange
transformation $gf$, by definition of $O$. Denote by $P_1, \dotsc, P_p$ the
successive intervals between pair-splitting singularities. Their number is
constant, but the position of the endpoints depends on $(\lambda,\mu)$.

\begin{lem}
\label{lem:nontrivial_comp} Fix $P=P_q$ for some $q\le p$. For almost every
$(\lambda,\mu)\in O$, there exists a nontrivial word $w_P \in F_2(x,y)$
(depending on $\lambda,\mu$) such that $w_P(f, g) = \id$ on $P$.
\end{lem}
\begin{proof}
Again, to simplify the notations, we will assume that $P=[0, u)$ with $u<1$.
The restriction of $\rho$ to $P$, that we denote by $\rho_P$, is a
permutation of $\{1,\dotsc, r\}$ for some $r$. It is irreducible by
Lemma~\ref{lem:irred_pair}. If $r=1$, then $gf=\id$ on $P$ and the conclusion
is trivial, taking $w_P=yx$. Assume now $r\ge 2$.

Let $\Phi : [0,u) \to [0,1)$ be the canonical rescaling map $\Phi(x) = x /
u$. Then $\tilde f \coloneqq \Phi \circ f_{|P} \circ \Phi^{-1}$ and $\tilde g
\coloneqq \Phi \circ g_{|P} \circ \Phi^{-1}$ are two interval exchange
transformations on $[0,1)$. To conclude, it suffices to show that they
satisfy a non-trivial relation. For this, we will apply
Theorem~\ref{thm:irred} to the pair $(\tilde g \tilde f, \tilde g)$. We have
to show that, for almost every pair $(\lambda, \mu)\in O$, then $\tilde
g\tilde f$ is uniquely ergodic and rigid by translations, and the pair
$(\tilde g\tilde f, \tilde g)$ is diophantine.

Denote by $(\nu_1, \dotsc, \nu_r)$ the lengths of the linearity intervals of
$gf$ in $P$. By Lemma~\ref{lem:ac}, their distribution is absolutely
continuous with respect to Lebesgue measure when $(\lambda,\mu)$ are
distributed according to Lebesgue measure in $O \subseteq \Delta_\pi \times
\Delta_\sigma$. The map
\begin{equation*}
  (\nu_1,\dotsc, \nu_r) \mapsto \pare*{\nu_1 / \pare*{\sum_{i=1}^r \nu_i}, \dotsc, \nu_r / \pare*{\sum_{i=1}^m \nu_i}},
\end{equation*}
taking values in the simplex $\Delta_{\rho_P}$, is smooth with a derivative
of maximal rank. Therefore, it maps Lebesgue measure to a measure absolutely
continuous with respect to Lebesgue measure. It follows that the interval
exchange transformation $\tilde g\tilde f$ is Lebesgue-generic. As $\rho_P$
is irreducible, we get from~\cite[Theorem~1.2]{veech_metric_1} that $\tilde g
\tilde f$ is almost surely uniquely ergodic, and from
Proposition~\ref{prop:rigid_transl} that it is almost surely rigid by
translations.

For almost every $(\lambda,\mu)$, the pair $(f, g)$ is diophantine, by
Proposition~\ref{prop:dioph}, for some constants $C$ and $c$. We will show
that in this case the pair $(\tilde g \tilde f, \tilde g)$ is diophantine
(for another constant $C'$, but the same $c$), which will conclude the proof.

Let $\mu'_1, \dotsc, \mu'_n$ be the lengths of the intervals of $g$ in $P$.
Note that they do not have to be a subset of the $\mu_i$, as some endpoints
of $P$ might be splitting singularities of $f$ and not of $g$. However, each
of them is a linear combination of the $\lambda_i$ and $\mu_j$.

We have to show that, for any integers $a_1,\dotsc, a_r, b_1,\dotsc, b_n$,
\begin{equation*}
  \abs*{\sum_{i=1}^{r} a_i \frac{\nu_i}{u} + \sum_{j=1}^{n} b_j \frac{\mu'_j}{u}}
  \geq C' / \max(\max_i \abs{a_i}, \max_j \abs{b_j})^c
\end{equation*}
whenever the expression on the left is nonzero. Multiply everything by $u$,
and note that the $\nu_i$ and the $\mu'_j$ are all integer linear
combinations of $(\lambda_1, \dotsc, \lambda_d, \mu_1, \dotsc, \mu_e)$. The
inequality to be shown becomes
\begin{equation*}
  \abs*{\sum_{i=1}^{d} A_i \lambda_i + \sum_{j=1}^{e} B_j \mu_j}
  \geq C' u / \max(\max_i \abs{a_i}, \max_j \abs{b_j})^c.
\end{equation*}
for some integers $A_i, B_j$ which are fixed linear combinations of $a_i$ and
$b_j$. The diophantine condition for $(f,g)$ implies that the left hand term
is bounded below by
\begin{equation*}
  C / \max(\max_i \abs{A_i}, \max_j \abs{B_j})^c.
\end{equation*}
As $\abs{A_i}$ and $\abs{B_j}$ are bounded by a constant multiple of
$\max(\max_i \abs{a_i}, \max_j \abs{b_j})$, the result follows.
\end{proof}

We can now conclude the proof of the main theorem. For almost every
$(\lambda,\mu) \in O$, we prove by induction that there is a nontrivial word
$w'_i$ such that $w'_i(f, g)=\id$ on $P_1 \cup \dotsb \cup P_i$, starting
from $w'_1=w_1$. Assuming that $w'_i$ is constructed, consider a nontrivial
word $w_{i+1}$ such that $w_{i+1}(f,g)=\id$ on $P_{i+1}$, as given by
Lemma~\ref{lem:nontrivial_comp}. Then $[w'_i(f, g), w_{i+1}(f, g)]=\id$ both
on $P_1 \cup \dotsb \cup P_i$ and $P_{i+1}$, as desired. The only issue is
that the word $[w'_i, w_{i+1}]$ could be trivial. However, the same argument
applies also to $[xw'_i x^{-1}, w_{i+1}]$ and to $[yw'_i y^{-1}, w_{i+1}]$ as
all the $P_j$ are invariant under $f$ and $g$. One of these words is
nontrivial by Lemma~\ref{lem:exists_nontrivial_conj}.

The word $w'_p$ constructed at the end of the induction satisfies the
conclusion of Theorem~\ref{main_thm}.

\bibliography{biblio}
\bibliographystyle{amsalpha}
\end{document}